\documentclass[12pt]{amsart}
\usepackage{amsfonts, amssymb, mathrsfs, amsmath}
\usepackage{fullpage, verbatim, bbm}
\usepackage[colorlinks=true]{hyperref}
\usepackage{breakurl}

\usepackage{etoolbox}
\makeatletter
\let\ams@starttoc\@starttoc
\makeatother
\usepackage[parfill]{parskip}
\makeatletter
\let\@starttoc\ams@starttoc
\patchcmd{\@starttoc}{\makeatletter}{\makeatletter\parskip\z@}{}{}
\makeatother

\newcommand\CA{{\mathscr A}} 
\newcommand\CB{{\mathscr B}}
 
\newcommand\CD{{\mathscr D}}

\newcommand\CI{{\mathcal I}}  
\newcommand\CAF{{\mathcal {AF}}} 
\newcommand\CIF{{\mathcal {IF}}} 
\newcommand\CDF{{\mathcal {DF}}} 
\newcommand\CIFM{{\mathcal {IFM}}} 
 
\newcommand\CRFM{{\mathcal {RFM}}} 
\newcommand\CAFM{{\mathcal {AFM}}}

\newcommand\BBK{{\mathbb K}}

\newcommand\BBQ{{\mathbb Q}}

\newcommand\BBZ{{\mathbb Z}}

\newcommand\Der{{\operatorname{Der}}}

\newcommand\pdeg{\operatorname{pdeg}}

\newcommand{\one}{\mathbbm{1}}

\numberwithin{equation}{section}

\theoremstyle{plain}
\newtheorem{lemma}[equation]{Lemma}
\newtheorem{theorem}[equation]{Theorem}

\newtheorem{corollary}[equation]{Corollary}
\newtheorem{proposition}[equation]{Proposition}
\theoremstyle{definition}
\newtheorem{defn}[equation]{Definition}
\newtheorem{remark}[equation]{Remark}
\newtheorem{remarks}[equation]{Remarks}
\newtheorem{example}[equation]{Example}

\begin{document}

\title[Inductive Freeness of Ziegler's Canonial Multiderivations]
{Inductive Freeness of 
Ziegler's Canonical Multiderivations}

\author[T.~Hoge and G.~R\"ohrle]{Torsten Hoge and Gerhard R\"ohrle}
\address
{Fakult\"at f\"ur Mathematik,
	Ruhr-Universit\"at Bochum,
	D-44780 Bochum, Germany}
\email{torsten.hoge@rub.de}
\email{gerhard.roehrle@rub.de}


\allowdisplaybreaks

\begin{abstract}
Let $\CA$ be a free hyperplane arrangement.
In 1989, Ziegler showed that 
the restriction $\CA''$ of $\CA$
to any hyperplane 
endowed with the natural multiplicity $\kappa$ 
is then a free multiarrangement $(\CA'',\kappa)$.
The aim of this paper is to prove
an analogue of Ziegler's theorem for the stronger notion of inductive freeness:  
if $\CA$ is inductively free, then so is 
the multiarrangement $(\CA'',\kappa)$.

In a related result
we derive that if a deletion $\CA'$ of $\CA$ is free and the corresponding restriction $\CA''$ is inductively free, then so is $(\CA'',\kappa)$ -- irrespective of the freeness of $\CA$. In addition, we show counterparts of the latter kind for additive and recursive freeness.
\end{abstract}

\maketitle

\setcounter{tocdepth}{1}
\tableofcontents


\section{Introduction}

The class of free arrangements, respectively free multiarrangments, 
plays a pivotal role in the theory of 
hyperplane arrangements, respectively 
multiarrangements. 
In \cite{ziegler:multiarrangements}, Ziegler 
introduced the notion of multiarrangements and initiated the study of their 
freeness.  
We begin by recalling Ziegler's 
fundamental construction from \emph{loc.~cit}. 

\begin{defn}
\label{def:kappa}
Let $\CA$ be an arrangement.
Fix $H_0 \in \CA$ and  consider the restriction 
$\CA''$ with respect to $H_0$.
Define the \emph{canonical multiplicity} 
$\kappa$ on $\CA''$ as follows. For $Y \in \CA''$ set 
\[
\kappa(Y) := |\CA_Y| -1,
\]
i.e., $\kappa(Y)$ is the number of hyperplanes in $\CA \setminus\{H_0\}$
lying above $Y$.
Ziegler showed that freeness of $\CA$ implies 
freeness of the multiarrangement $(\CA'', \kappa)$.
We also call $\kappa$ the \emph{Ziegler multiplicity} and 
$(\CA'', \kappa)$ the \emph{Ziegler restriction} of $\CA$ on $H_0$.
\end{defn}

\begin{theorem}
[{\cite[Thm.~11]{ziegler:multiarrangements}}]
\label{thm:zieglermulti}
Let $\CA$ be a free arrangement with exponents
$\exp \CA = \{1, e_2, \ldots, e_\ell\}$.
Let $H_0 \in \CA$ and consider the restriction 
$\CA''$ with respect to $H_0$.
Then the multiarrangement $(\CA'', \kappa)$ is free with
exponents
$\exp (\CA'', \kappa) = \{e_2, \ldots, e_\ell\}$. 
\end{theorem}

Because of the relevance of Ziegler's multiplicity
in the theory of free arrangements, 
it is natural to investigate stronger
freeness properties for $(\CA'', \kappa)$ and
specifically to ask for an analogue of
Theorem \ref{thm:zieglermulti} for 
inductive freeness. This is the content of 
our main result. 

\begin{theorem}
	\label{thm:main}
	If $\CA$ is inductively free, then so is any Ziegler restriction $(\CA'',\kappa)$ of $\CA$.
\end{theorem}

In \cite{hogeroehrle:Ziegler}, we
initialized the study 
of the implication of Theorem \ref{thm:main}
and examined the multiarrangements $(\CA'', \kappa)$ where 
the underlying
class of arrangements $\CA$
consists of reflection arrangements of complex reflection groups.

\begin{remark}
While the multiset of exponents of a free simple arrangement $\CA$ is a combinatorial invariant of $\CA$, i.e., $\exp \CA $ only depends on the intersection lattice of $\CA$, 
thanks to Terao's seminal factorization theorem, cf.~\cite[Thm.~4.137]{orlikterao:arrangements},
this is no longer the case when we pass to multiarrangements, as Ziegler already observed in \cite[Prop.~10]{ziegler:multiarrangements}. Thus, unsurprisingly,
while inductive freeness is a combinatorial 
property for simple arrangements, e.g.~see \cite[Lem.~2.5]{cuntzhoge}, in contrast,
for multiarrangements it is not combinatorial, see \cite[Ex.~2.4]{dipasqualewakefield}.
Moreover, 
despite the fact that the Ziegler multiplicity $\kappa$ only depends on the intersection lattice of $\CA$, 
the freeness of $(\CA'',\kappa)$ is not combinatorial, 
as demonstrated by the following example, which is based on the so called whirl matroid, cf.~\cite{dipasqualewakefield}.
Indeed, the example demonstrates
that the freeness of a Ziegler restriction is not combinatorial in a strong sense, for it neither depends on the intersection lattice of the underlying ambient arrangement nor on that of the restriction.
So the assertion of Theorem \ref{thm:main} is rather unexpected. 
\end{remark}

\begin{example}
	\label{ex:noncombk}
	We define two rational rank $4$ arrangements $\CA$ and $\CB$  which have the $12$ hyperplanes given by the following linear forms 
	\begin{equation*}
		x+5t, x+6t, x-5t, y+2t, y-4t, y-t, z+4t, z+3t, z-3t,
		x+z+7t, y+z-6t, t
	\end{equation*}
	in $\BBQ[x,y,z,t]$ in common, and $\CA$ has additionally the hyperplane given by $x+y-2t$ and $\CB$ the one given by $x-2y -2t$. So $\CA$ and $\CB$ only differ by one hyperplane.
	The intersection lattices of $\CA$ and $\CB$ are identical.
	
	The Ziegler restrictions of $\CA$ and $\CB$  to $H_0 = \ker(t)$ are given by
	\begin{equation*}
		\label{eq:a}
		Q(\CA'', \kappa) = x^3y^3z^3(x+y)(x+z)(y+z)
	\end{equation*}
	and 
	\begin{equation*}
		\label{eq:b}
		Q(\CB'', \kappa) = x^3y^3z^3(x-2y)(x+z)(y+z),
	\end{equation*}
	respectively. Consequently, thanks to 
	\cite[Thm.~4.2]{dipasqualewakefield}, 
	we get that $(\CB'',\kappa)$ is free, but $(\CA'',\kappa)$ is not.
	Moreover, the intersection lattices of $\CA''$ and $\CB''$ are identical as well.
	Finally, note that none of the simple arrangements $\CA$, $\CB$, $\CA''$, and $\CB''$ is free, see \cite[Prop.~5.2]{dipasqualewakefield}. 
\end{example}

Note that even if $\CA$ is inductively free, there might be 
restrictions $\CA^{H_0}$ of $\CA$ which are themselves not even free, e.g.~see \cite[Rem.~3.6]{amendmoellerroehrle:aspherical}. Nevertheless, as asserted by Theorem \ref{thm:main}, the Ziegler restriction $(\CA^{H_0},\kappa)$ is again inductively free for \emph{every} $H_0 \in \CA$.

There is also the weaker notion of \emph{recursive freeness}, both for simple and also for  multiarrangements, see \cite[Def.~4.60]{orlikterao:arrangements} and Definition \ref{def:recfree} (\cite[Def.~2.21]{hogeroehrleschauenburg:free}), respectively.
We show the following analogue of Theorem \ref{thm:main} in this context.

\begin{theorem}
	\label{thm:main5}
	If $\CA$ is recursively free, then so is any Ziegler restriction $(\CA'',\kappa)$ of $\CA$.
\end{theorem}

Even if $\CA$ itself is not inductively free, we derive the following result still guaranteeing 
inductive freeness  of $(\CA^{H_0},\kappa)$, provided  the restriction $\CA^{H_0}$ is still inductively free; likewise for additive and recursive freeness. For the notion of \emph{additive freeness} for simple and multiarrangements, see Definitions \ref{def:addfree} and \ref{def:multaddfree}, respectively.

\begin{theorem}
	\label{thm:main2}
	Let $H_0 \in \CA$. Suppose that $\CA\setminus \{H_0\}$ is free.
	If $\CA^{H_0}$ is inductively free (resp.~additively free, resp.~recursively free), then so is $(\CA^{H_0},\kappa)$.
\end{theorem}

Note that one cannot drop the freeness condition on $\CA\setminus \{H_0\}$ in 
Theorem \ref{thm:main2}; see \cite[Ex.~2.24]{hogeroehrle:Ziegler} for an example where 
$\CA^{H_0}$ is inductively free while both $\CA\setminus \{H_0\}$ and 
$(\CA^{H_0},\kappa)$ are not even free.
While in \emph{loc.~cit.} 
both $\CA$ and $\CA'$ are non-free,
Example \ref{ex:failure1.5} illustrates the failure of Theorem \ref{thm:main2} when 
$\CA$ is still free and $\CA'$ is not.
It turns out that Theorem \ref{thm:main2} is a consequence of  Theorem \ref{thm:main} (resp.~Theorem \ref{thm:main5}), see \S\ref{Sec:proofs}.

The $3$-arrangement $\CA$ given by $Q(\CA) = xyz(x+y+z)$ is the smallest example for Theorem \ref{thm:main2} 
when $\CA$ is non-free, \cite[Ex.~4.34]{orlikterao:arrangements}.

The following is immediate from Theorems \ref{thm:main2} and \ref{thm:add-del-simple}.

\begin{corollary}
	\label{cor:main2}
	Let $H_0 \in \CA$. Suppose both $\CA$ and $\CA^{H_0}$ are free and $\exp \CA^{H_0} \subset \exp \CA$. If $\CA^{H_0}$ is inductively free (resp.~additively free, resp.~recursively free), 
	then so is $(\CA^{H_0},\kappa)$.
\end{corollary}

Surprisingly, the analogue of Theorem \ref{thm:main} in the weaker setting of additive freeness turns out to be false. A particular subarrangement in a restriction of the Weyl arrangement of type $E_7$ constructed in  
\cite{hogeroehrle:stairfree} which is additively free but not inductively free provides an example of this kind, see Example \ref{ex:additive}. 

\begin{theorem}
	\label{thm:main6}
	There are arrangements $\CA$ which are additively free such that for some $H_0 \in \CA$ the Ziegler restriction $(\CA^{H_0},\kappa)$ is not additively free.
\end{theorem}

We briefly discuss some immediate consequences of the theorems above.

In \cite[Thm.~1.4]{hogeroehrle:Ziegler}, we gave a complete classification of 
all complex reflection arrangements $\CA = \CA(W)$, for $W$ a complex reflection group,  
so that $(\CA^{H_0},\kappa)$ is inductively free, for every $H_0 \in \CA$. In part this involved carrying out 
various intricate induction tables or complicated arguments, depending on computer calculations.  
With the aid of Theorem \ref{thm:main} the proof can be simplified and shorted considerably, as illustrated by the following consequence of  Theorem \ref{thm:main}.
Let $W$ be a Coxeter group and let  $\CA(W)$ be the associated Coxeter arrangement.
Coxeter arrangements $\CA(W)$ are known to be inductively free, owing to \cite{cuntz:indfree}. Therefore, 
thanks to Theorem \ref{thm:main}, we immediately recover the following 
without any further case by case analysis and free of any computer computations.

\begin{corollary}
	\label{cor:coxeter}
	Let $\CA = \CA(W)$ be a Coxeter arrangement. Then every Ziegler restriction $(\CA'',\kappa)$ of $\CA$ is inductively free.
\end{corollary}

The proof of Corollary \ref{cor:coxeter} in the case 
when $W$ is of type $E_8$ in \cite{hogeroehrle:Ziegler} is rather involved,  entailing intricate computer calculations.

In \cite[Thm.~5.10]{abeteraowakefield:euler},
certain multiplicities defined on supersolvable arrangements were shown to be 
inductively free. 
Since supersolvable arrangements are inductively free \cite[Thm.~4.58]{orlikterao:arrangements}, we immediately get the following consequence from Theorem \ref{thm:main}. 

\begin{corollary}
	\label{cor:super}
	For $\CA$ supersolvable, any Ziegler restriction $(\CA'',\kappa)$ 
	is inductively free.
\end{corollary}

The reflection arrangement 
of the complex reflection group $G(r,p,\ell)$
is supersolvable for all $r, \ell \ge 2$ and $p \ne r$ if $\ell \ge 3$, cf.~\cite[Thm.~1.2]{hogeroehrle:super}. As part of the classification in \cite[Thm.~1.4]{hogeroehrle:Ziegler}, we showed the following, which now follows readily from Corollary \ref{cor:super}.

\begin{corollary}
	\label{cor:grrl}
	Let $\CA = \CA(G(r,p,\ell))$ for $r, \ell \ge 2$ and $p \ne r$  if $\ell \ge 3$. Then  $(\CA^{H_0},\kappa)$ is inductively free, for every $H_0\in\CA$.
\end{corollary}
 

The so called 
\emph{arrangements of ideal type}
$\CA_\CI$ are certain natural subarrangements 
of Weyl arrangements associated with 
ideals in the set of positive roots of 
a reduced root system.
They are known to be free due to work of Sommers and Tymoczko \cite{ST06} and 
Abe et al \cite{p-ABCHT-14}. In particular, in \emph{loc.~cit}, the $\CA_\CI$ were shown to be free 
in a uniform fashion where the exponents are given in a combinatorial manner.
This in turn lead to the question whether all $\CA_\CI$ do satisfy the stronger property of inductive freeness.
This was settled affirmatively in 
\cite{gR17} and \cite{cuntzroehrleschauenburg:ideal}.
Thus  Theorem \ref{thm:main} immediately entails the following.

\begin{corollary}
	\label{cor:ideal}
Let $\CA_\CI$ be an ideal subarrangement of the Weyl
arrangement $\CA(W)$. 
Then every Ziegler restriction $(\CA_\CI'',\kappa)$ of $\CA_\CI$ is inductively free.
\end{corollary}

Note that in general, if $(\CA, \mu)$ is inductively free (resp.~additively free, resp.~recursively free), it need not follow that $\CA$ itself is inductively free (resp.~additively free, resp.~recursively free). That is, while there is an inductive (resp.~additive, resp.~recursive) chain of free subarrangements from the empty arrangement to $(\CA, \mu)$ (cf.~Remark \ref{rem:recchain}, Definition \ref{def:multaddfree}), such a chain need not pass through the 
simple arrangement $\CA$ itself.
We illustrate this by our next example.

\begin{example}
	\label{ex:indfree-recfree}	
	Thanks to Corollary \ref{cor:ideal},
	for an ideal arrangement $\CA_\CI$, any Ziegler restriction    
	$(\CA''_\CI, \kappa)$ of $\CA_\CI$ is inductively free. 
	However, 
	there are certain ideal arrangements $\CA_\CI$ in Weyl arrangements of type $D$ which admit restrictions $\CA''_\CI$ that are not free, see \cite[Rem.~3.6]{amendmoellerroehrle:aspherical}. It thus follows that in each of these instances the inductively free Ziegler restriction $(\CA''_\CI, \kappa)$ does not admit a chain of free subarrangements passing through the simple arrangement $\CA''_\CI$, as the latter itself is not free.
\end{example}

\begin{example}
	\label{ex:shi}
Thanks to work of Edelman and Reiner \cite[Thm.~3.2]{ER96_FreeArrRhombicTilings} and Athanasiadis \cite[Thm.~3.3]{Atha98}, respectively, the cones of the extended 
Catalan arrangements in type $A_\ell$ and cones of the extended 
Shi arrangements in type $A_\ell$ are inductively free.
As a consequence of Theorem \ref{thm:main}, the multiarrangements $(\CA(A_\ell), \mu)$ are inductively free, where $\CA(A_\ell)$ is the reflection arrangement of type $A_\ell$ and $\mu = c \cdot \one$, for $c \in \BBZ_{\ge 1}$, is a constant multiplicity on $\CA(A_\ell)$. The latter fact was proved in \cite[Thm.~1.1]{conradroehrle:indfree} in an elaborate induction argument.
\end{example}

We continue with some applications of Theorems \ref{thm:main5} and \ref{thm:main2}.

\begin{example}
	\label{ex:g31}
	Let $\CA = \CA(G_{31})$ be the reflection arrangement of the complex reflection group of type $G_{31}$. 
	Note that while $\CA$ is not inductively free, by  \cite[Lem.~3.5]{hogeroehrle:inductivelyfree},
	$\CA^{H_0}$ is inductively free for each $H_0 \in \CA$, thanks to \cite[Lem.~4.1]{amendhogeroehrle:indfree}. 
	Since $\exp \CA^{H_0} \subset \exp \CA$, owing to \cite[Table C.12]{orlikterao:arrangements}, it follows from Corollary \ref{cor:main2} that 
	 $(\CA^{H_0},\kappa)$ is inductively free, for every $H_0 \in \CA$. This was first shown in  \cite[Cor.~3.14]{hogeroehrle:Ziegler} by more complicated means. 
	 
	 This example is going to make another appearance in \S \ref{Yconverse} below in connection with the question of a converse to Theorem \ref{thm:main}.
\end{example}

\begin{example}
	\label{ex:akl}
	Orlik and Solomon defined so called intermediate arrangements $\CA^k_\ell(r)$,  
cf.~\cite[\S 6.4]{orlikterao:arrangements}, which
interpolate between the
reflection arrangements of the imprimitive complex reflection groups $G(r,r,\ell)$ and $G(r,1,\ell)$. 
They show up as restrictions of the reflection arrangement
of $G(r,r,n)$, 
cf.~\cite[Prop.\ 6.84]{orlikterao:arrangements}.

For $r, \ell \geq 2$ and $0 \leq k \leq \ell$, the defining polynomial of
$\CA^k_\ell(r)$ is given by
\[
Q(\CA^k_\ell(r)) = x_1 \cdots x_k \prod\limits_{1 \leq i < j \leq \ell}(x_i^r - x_j^r),
\]
so that 
$\CA^\ell_\ell(r) = \CA(G(r,1,\ell))$ and 
$\CA^0_\ell(r) = \CA(G(r,r,\ell))$. 
For $k \neq 0, \ell$, these are not reflection arrangements
themselves. 

Let $\CA = \CA^k_\ell(r)$ for $1 \leq k \leq \ell$. 
First let $H_0$ be a coordinate hyperplane in $\CA$. 
Then, owing to \cite[Prop.\ 6.82]{orlikterao:arrangements},  $\CA^{H_0} \cong \CA^{\ell-1}_{\ell-1}(r) = \CA(G(r,1,\ell-1))$. The latter is inductively free, by \cite[Thm.~1.1]{hogeroehrle:inductivelyfree}. 
Since $\CA' \cong \CA^{k-1}_\ell(r)$ is also free, 
it follows from Theorem \ref{thm:main2} that 
$(\CA^{H_0},\kappa)$ is also inductively free.
Observe that $\CA$ itself is \emph{not} inductively free, for $r, \ell  \ge 3$ and $0 \leq k \leq \ell -3$, due to  \cite[Lem.~4.1]{amendhogeroehrle:indfree}.

Next suppose $\ell \ge 5$ and let $H_0 = \ker(x_1-x_2)$. Then $(\CA^{H_0},\kappa)$ is no longer inductively free, according to \cite[Prop.~4.2]{hogeroehrleschauenburg:free}.
Nevertheless, by \cite[Thm.~3.6(ii)]{amendhogeroehrle:indfree}, $\CA = \CA^k_\ell(r)$ is recursively free for $r, \ell  \ge 3$ and any $0 \leq k \leq \ell$. It follows from Theorem \ref{thm:main5} that 
$(\CA^{H_0},\kappa)$ is still recursively free, for every $H_0 \in \CA$. 
\end{example}

In view of Example \ref{ex:akl}, with the aid of Theorems \ref{thm:main} and \ref{thm:main2}, it is feasible to obtain a 
classification of all instances when the Ziegler restriction is inductively free where the underlying simple arrangement is a 
restriction of a complex reflection arrangement $\CA(W)$; this is carried out in   
 \cite{hogeroehrle:ZieglerIII}. 

The paper is organized as follows.
In Section \ref{ssect:hyper} we recall 
the fundamental results for free arrangements, 
in particular Terao's 
Addition Deletion Theorem \ref{thm:add-del-simple}
and the subsequent
notion of an inductively free arrangement.
In Definition \ref{def:addfree} we recapitulate Abe's weaker notion 
of additive freeness, cf.~\cite[Def.~1.6]{abe:sf}.
In \S \ref{ssec:multi} and the following sections we 
introduce the concept of multiarrangements and basic properties of their freeness. In particular, we recall the 
fundamental Addition-Deletion Theorem  \ref{thm:add-del} in this setting due to Abe, Terao and Wakefield, 
\cite[Thm.~0.8]{abeteraowakefield:euler}.

In \S \ref{ssec:inductive}, \S\ref{ssec:recursive}, and \S\ref{ssec:additive}, we introduce the notions of inductive, recursive and additive freeness for multiarrangements, respectively. 
Section \ref{sec:filtrations} is devoted to  
the concept of free filtrations. Apart from one crucial new result, Lemma \ref{lemma:free_sequence_to_simple_multiplicity},
here we mainly recall facts delineated in 
\cite[\S3]{hogeroehrle:Ziegler}.

The proofs of Theorems \ref{thm:main} to \ref{thm:main6}
are carried out in Section \ref{Sec:proofs}. Finally, we discuss some complements to our main developments in the last section.

\section{Recollections and Preliminaries}
\label{sect:prelim}

\subsection{}
\label{ssect:hyper}
Let $V = \BBK^\ell$ 
be an $\ell$-dimensional $\BBK$-vector space.
A \emph{hyperplane arrangement} is a pair
$(\CA, V)$, where $\CA$ is a finite collection of hyperplanes in $V$.
Usually, we simply write $\CA$ in place of $(\CA, V)$.
We write $|\CA|$ for the number of hyperplanes in $\CA$.
The empty arrangement in $V$ is denoted by $\Phi_\ell$.

\subsection{}
The \emph{lattice} $L(\CA)$ of $\CA$ is the set of subspaces of $V$ of
the form $H_1\cap \dotsm \cap H_i$ where $\{ H_1, \ldots, H_i\}$ is a subset
of $\CA$. 
For $X \in L(\CA)$, we have two associated arrangements, 
firstly
$\CA_X :=\{H \in \CA \mid X \subseteq H\} \subseteq \CA$,
the \emph{localization of $\CA$ at $X$}, 
and secondly, 
the \emph{restriction of $\CA$ to $X$}, $(\CA^X,X)$, where 
$\CA^X := \{ X \cap H \mid H \in \CA \setminus \CA_X\}$.
Note that $V$ belongs to $L(\CA)$
as the intersection of the empty 
collection of hyperplanes and $\CA^V = \CA$. 
The lattice $L(\CA)$ is a partially ordered set by reverse inclusion:
$X \le Y$ provided $Y \subseteq X$ for $X,Y \in L(\CA)$.

Suppose $\CA \neq \Phi_\ell$. Fix a member $H_0$ in $\CA$. Set $\CA' = \CA \setminus \{H_0\}$ and  $\CA'' = \CA^{H_0}$. Then $(\CA, \CA', \CA'')$ is frequently referred to as a \emph{triple of arrangements}. 

\subsection{}
\label{ssect:free}
Let $S = S(V^*)$ be the symmetric algebra of the dual space $V^*$ of $V$.
If $x_1, \ldots , x_\ell$ is a basis of $V^*$, then we identify $S$ with 
the polynomial ring $\BBK[x_1, \ldots , x_\ell]$.
Letting $S_p$ denote the $\BBK$-subspace of $S$
consisting of the homogeneous polynomials of degree $p$ (along with $0$),
$S$ is naturally $\BBZ$-graded: $S = \oplus_{p \in \BBZ}S_p$, where
$S_p = 0$ in case $p < 0$.

Let $\Der(S)$ be the $S$-module of algebraic $\BBK$-derivations of $S$.
Using the $\BBZ$-grading on $S$, $\Der(S)$ becomes a graded $S$-module.
For $i = 1, \ldots, \ell$, 
let $D_i := \partial/\partial x_i$ be the usual derivation of $S$.
Then $D_1, \ldots, D_\ell$ is an $S$-basis of $\Der(S)$.
We say that $\theta \in \Der(S)$ is 
\emph{homogeneous of polynomial degree p}
provided 
$\theta = \sum_{i=1}^\ell f_i D_i$, 
where $f_i$ is either $0$ or homogeneous of degree $p$
for each $1 \le i \le \ell$.
In this case we write $\pdeg \theta = p$.

Let $\CA$ be an arrangement in $V$. 
Then for $H \in \CA$ we fix $\alpha_H \in V^*$ with
$H = \ker(\alpha_H)$.
The \emph{defining polynomial} $Q(\CA)$ of $\CA$ is given by 
$Q(\CA) := \prod_{H \in \CA} \alpha_H \in S$.

The \emph{module of $\CA$-derivations} of $\CA$ is 
defined by 
\[
D(\CA) := \{\theta \in \Der(S) \mid \theta(\alpha_H) \in \alpha_H S
\text{ for each } H \in \CA \} .
\]
We say that $\CA$ is \emph{free} if the module of $\CA$-derivations
$D(\CA)$ is a free $S$-module.

With the $\BBZ$-grading of $\Der(S)$, 
$D(\CA)$ also 
becomes a graded $S$-module,
\cite[Prop.~4.10]{orlikterao:arrangements}.
If $\CA$ is a free arrangement, then the $S$-module 
$D(\CA)$ admits a basis of $\ell$ homogeneous derivations, 
say $\theta_1, \ldots, \theta_\ell$, \cite[Prop.~4.18]{orlikterao:arrangements}.
While the $\theta_i$'s are not unique, their polynomial 
degrees $\pdeg \theta_i$ 
are unique (up to ordering). This multiset is the set of 
\emph{exponents} of the free arrangement $\CA$
and is denoted by $\exp \CA$.

\subsection{}

The fundamental \emph{Addition Deletion Theorem} 
due to Terao  \cite{terao:freeI} plays a 
crucial role in the study of free arrangements, 
\cite[Thm.~4.51]{orlikterao:arrangements}.

\begin{theorem}
\label{thm:add-del-simple}
Suppose $\CA \neq \Phi_\ell$ and
let $(\CA, \CA', \CA'')$ be a triple of arrangements. Then any 
two of the following statements imply the third:
\begin{itemize}
\item[(i)] $\CA$ is free with $\exp\CA = \{ b_1, \ldots , b_{\ell -1}, b_\ell\}$;
\item[(ii)] $\CA'$ is free with $\exp\CA' = \{ b_1, \ldots , b_{\ell -1}, b_\ell-1\}$;
\item[(iii)] $\CA''$ is free with $\exp\CA'' = \{ b_1, \ldots , b_{\ell -1}\}$.
\end{itemize}
\end{theorem}

\subsection{}

Theorem \ref{thm:add-del-simple} motivates the following notion.

\begin{defn}
	[{\cite[Def.~4.53]{orlikterao:arrangements}}]
\label{def:indfree-simple}
The class $\CIF$ of \emph{inductively free} arrangements 
is the smallest class of arrangements subject to
\begin{itemize}
\item[(i)] $\Phi_\ell \in \CIF$ for each $\ell \ge 0$;
\item[(ii)] if there exists a hyperplane $H_0 \in \CA$ such that both
$\CA'$ and $\CA''$ belong to $\CIF$, and $\exp \CA '' \subseteq \exp \CA'$, 
then $\CA$ also belongs to $\CIF$.
\end{itemize}
\end{defn}

\subsection{}

In view of Theorem \ref{thm:add-del-simple}, it is also natural to consider the following property, \cite[Def.~1.6]{abe:sf} which is weaker than inductive freeness. Note, this definition is equivalent to \emph{loc.~cit.}, cf.~\cite[Rem.~1.7]{abe:sf}.

\begin{defn}
	\label{def:addfree}
	An arrangement $\CA$ is called 
	\emph{additively free} if there is a \emph{free filtration} 
	\[
	\Phi_\ell = \CA_0 \subset \CA_1 \subset \cdots \subset \CA_n = \CA,
	\]
	of $\CA$, i.e., where each $\CA_i$ is free with $|\CA_i| = i$ for each $i$. 
	Denote this class by $\CAF$.
\end{defn}

\begin{remark}
	In view of  \cite[Thm.~4.46]{orlikterao:arrangements}, we have $\CIF \subseteq \CAF$.
	However, there are examples of  arrangements  which are additively free but not inductively free, see \cite{hogeroehrle:stairfree}.
\end{remark}

\subsection{}
\label{ssec:multi}
A \emph{multiarrangement}  is a pair
$(\CA, \mu)$ consisting of a hyperplane arrangement $\CA$ and a 
\emph{multiplicity} function
$\mu : \CA \to \BBZ_{\ge 0}$ associating 
to each hyperplane $H$ in $\CA$ a non-negative integer $\mu(H)$.
The \emph{order} of the multiarrangement $(\CA, \mu)$ 
is defined by 
$|\mu| := |(\CA, \mu)| = \sum_{H \in \CA} \mu(H)$.
The \emph{defining polynomial} $Q(\CA, \mu)$ 
of the multiarrangement $(\CA, \mu)$ is given by 
\[
Q(\CA, \mu) := \prod_{H \in \CA} \alpha_H^{\mu(H)},
\] 
a polynomial of degree $|\mu|$ in $S$, where $\alpha_H \in V^*$ such that $H = \ker \alpha_H$.

For a multiarrangement $(\CA, \mu)$, the underlying 
arrangement $\CA$ is sometimes called the associated 
\emph{simple} arrangement, and so $(\CA, \mu)$ itself is  
simple if and only if $\mu(H) = 1$ for each $H \in \CA$. 
In the sequel, we denote by $\one$ the simple multiplicity on a given arrangement. 

Let $(\CA, \mu)$ be a multiarrangement in $V$ and let 
$X \in L(\CA)$. The 
\emph{localization of $(\CA, \mu)$ at $X$} is defined to be $(\CA_X, \mu_X)$,
where $\mu_X = \mu |_{\CA_X}$.

\subsection{}
\label{ssec:freemulti}
Following Ziegler \cite{ziegler:multiarrangements},
we extend the notion of freeness to multiarrangements as follows.
Let $\alpha_H \in V^*$ such that $H = \ker \alpha_H$ for each $H \in \CA$.
The \emph{module of $\CA$-derivations} of $(\CA, \mu)$ is 
defined by 
\[
D(\CA, \mu) := \{\theta \in \Der(S) \mid \theta(\alpha_H) \in \alpha_H^{\mu(H)} S 
\text{ for each } H \in \CA\}.
\]
We say that $(\CA, \mu)$ is \emph{free} if 
$D(\CA, \mu)$ is a free $S$-module, 
\cite[Def.~6]{ziegler:multiarrangements}.

As in the case of simple arrangements,
$D(\CA, \mu)$ is a $\BBZ$-graded $S$-module and 
thus, if $(\CA, \mu)$ is free, there is a 
homogeneous basis $\theta_1, \ldots, \theta_\ell$ of $D(\CA, \mu)$.
The multiset of the unique polynomial degrees $\pdeg \theta_i$ 
forms the set of \emph{exponents} of the free multiarrangement $(\CA, \mu)$
and is denoted by $\exp (\CA, \mu)$.

Next we recall Ziegler's analogue of Saito's criterion.

\begin{theorem}
	[{\cite[Thm.~8]{ziegler:multiarrangements}}]
	\label{thm:ziegler-saito}
	For $\theta_1, \dots, \theta_\ell$ in $D(\CA, \mu)$, the following are equivalent:
	\begin{itemize}
		\item[(i)]  $\{\theta_1, \dots, \theta_\ell\}$ is a basis of $D(\CA, \mu)$,
		\item[(ii)] $	\det M(\theta_1, \dots, \theta_\ell)\ \dot=\ Q(\CA, \mu)$.
	\end{itemize}
\end{theorem}

Here the notation $\dot=$ indicates equality
up to a non-zero scalar multiple and 
$M(\theta_1, \dots, \theta_\ell)$ is the coefficient matrix of 
$\{\theta_1, \dots, \theta_\ell\}$, cf.~\cite[Def.~4.11]{orlikterao:arrangements}.

\subsection{}
We recall the construction from \cite{abeteraowakefield:euler} for the 
counterpart of Theorem \ref{thm:add-del-simple} in this more general setting.

\begin{defn}
\label{def:Euler}
Let $(\CA, \mu) \ne \Phi_\ell$ be a multiarrangement. Fix $H_0$ in $\CA$.
We define the \emph{deletion}  $(\CA', \mu')$ and \emph{restriction} $(\CA'', \mu^*)$
of $(\CA, \mu)$ with respect to $H_0$ as follows.
If $\mu(H_0) = 1$, then set $\CA' = \CA \setminus \{H_0\}$
and define $\mu'(H) = \mu(H)$ for all $H \in \CA'$.
If $\mu(H_0) > 1$, then set $\CA' = \CA$
and define $\mu'(H_0) = \mu(H_0)-1$ and
$\mu'(H) = \mu(H)$ for all $H \ne H_0$.

Let $\CA'' = \{ H \cap H_0 \mid H \in \CA \setminus \{H_0\}\ \}$.
The \emph{Euler multiplicity} $\mu^*$ of $\CA''$ is defined as follows.
Let $Y \in \CA''$. Since the localization $\CA_Y$ is of rank $2$, the
multiarrangement $(\CA_Y, \mu_Y)$ is free, 
\cite[Cor.~7]{ziegler:multiarrangements}. 
According to 
\cite[Prop.~2.1]{abeteraowakefield:euler},
the module of derivations 
$D(\CA_Y, \mu_Y)$ admits a particular homogeneous basis
$\{\theta_Y, \psi_Y, D_3, \ldots, D_\ell\}$,
such that $\theta_Y \notin \alpha_0 \Der(S)$
and $\psi_Y \in \alpha_0 \Der(S)$,
where $H_0 = \ker \alpha_0$.
Then on $Y$ the Euler multiplicity $\mu^*$ is defined
to be $\mu^*(Y) = \pdeg \theta_Y$.

Often, 
$(\CA, \mu), (\CA', \mu')$ and $(\CA'', \mu^*)$ 
is referred to as the \emph{triple} of 
$(\CA, \mu)$ with respect to $H_0$. 
\end{defn}

We require some core results from 
\cite{abeteraowakefield:euler}.  

\begin{theorem}
[{\cite[Thm.~0.4]{abeteraowakefield:euler}}]
\label{thm:restriction}
Suppose that $(\CA, \mu) \ne \Phi_\ell$.
Fix $H_0$ in $\CA$. If both 
$(\CA, \mu)$ and $(\CA', \mu')$ are free, 
then there exists a basis 
$\{\theta_1, \ldots, \theta_\ell \}$ of 
$D(\CA', \mu')$ such that 
$\{\theta_1, \ldots, \alpha_k \theta_k, \ldots, \theta_\ell \}$
is a basis of $D(\CA, \mu)$ 
for some $1 \le k \le \ell$.
\end{theorem}

\begin{defn}
		\label{def:Sbar}
		Fix $H_0 = \ker \alpha_0$ in $\CA$ and let $\CA''$ be the restriction with respect to $H_0$.
		Consider the projection $S \to \overline S  := S/\alpha_0 S$, 
		$f \mapsto \overline f$.
		There is 
		a canonical restriction map 
		\[
		\rho : D(\CA,\mu) \to D(\CA'', \mu^*), \quad \theta \mapsto \overline \theta
		\]  
		where $\overline \theta(\overline f) := \overline {\theta(f)}$, for 
		$f \in S$.	
		For $\mu \equiv \one$, we get $\mu^* \equiv \one$ (cf.~\cite[Rem.~03]{abeteraowakefield:euler}) and so $\rho$ specializes to the well known restriction map from \cite[Prop.~4.45]{orlikterao:arrangements}
		\begin{equation}
		\label{eq:rho}
		\rho : D(\CA) \to D(\CA'')
		\end{equation}
		 in case of simple arrangements.
\end{defn}

\begin{theorem}
[{\cite[Thm.~0.6]{abeteraowakefield:euler}}]
\label{thm:restriction2}
Suppose that $(\CA, \mu) \ne \Phi_\ell$.
Fix $H_0$ in $\CA$. 
Suppose that both 
$(\CA, \mu)$ and $(\CA', \mu')$ are free.
Let 
$\{\theta_1, \ldots, \theta_\ell \}$ be a basis of 
$D(\CA', \mu')$ 
as in Theorem \ref{thm:restriction}.
Then 
$\{\overline \theta_1, \ldots, \overline \theta_{k-1}, \overline \theta_{k+1}, \ldots, \overline \theta_\ell \}$
is a basis of $D(\CA'', \mu^*)$. 
\end{theorem}

Here is the counterpart of Theorem \ref{thm:add-del-simple} for multiarrangements.

\begin{theorem}
[{\cite[Thm.~0.8]{abeteraowakefield:euler}}
Addition Deletion Theorem for Multiarrangements]
\label{thm:add-del}
Suppose that $(\CA, \mu) \ne \Phi_\ell$.
Fix $H_0$ in $\CA$ and 
let  $(\CA, \mu), (\CA', \mu')$ and  $(\CA'', \mu^*)$ be the triple with respect to $H_0$. 
Then any  two of the following statements imply the third:
\begin{itemize}
\item[(i)] $(\CA, \mu)$ is free with $\exp (\CA, \mu) = \{ b_1, \ldots , b_{\ell -1}, b_\ell\}$;
\item[(ii)] $(\CA', \mu')$ is free with $\exp (\CA', \mu') = \{ b_1, \ldots , b_{\ell -1}, b_\ell-1\}$;
\item[(iii)] $(\CA'', \mu^*)$ is free with $\exp (\CA'', \mu^*) = \{ b_1, \ldots , b_{\ell -1}\}$.
\end{itemize}
\end{theorem}

\subsection{}

In our following rather general result, we consider an arbitrary chain of free subarrangements of a free arrangement. For a restriction to a hyperplane stemming from the smallest of these free subarrangements  
we then investigate the resulting chain of multiarrangements where we endow each restriction with the canonical multiplicity. 
It turns out that 
the Euler multiplicity at each step for the associated restricted multiarrangement is just itself again another Ziegler multiplicity. This result turns out to be extremely useful in our sequel in determining inductively free Ziegler restrictions.

\begin{lemma}
	\label{lemma:free_sequence_to_ziegler_multiplicity}
	Let $\CA$ be an $\ell$-arrangement and let $H_1$,\ldots,$H_n$ be distinct hyperplanes in $\CA$.
	Define $\CA_0 := \CA \setminus \{H_1,\ldots,H_n\}$, and $\CA_i := \CA_{i-1} \cup \{H_i\}$ for $i = 1, \ldots, n$. 
	Suppose that $\CA_0,\ldots,\CA_n = \CA$ are all free.  Then for a fixed $H \in \CA_0$, we have the following:
	\begin{enumerate}
		\item 
		$(\CA_i^H,\kappa_i)$ is free for $i=0,\ldots,n$, where $\kappa_i$ is the Ziegler multiplicity of $\CA_i^H$.
		
		\item Consider the triple $(\CA_{i}^H,\kappa_i)$, $(\CA_{i-1}^H,\kappa_{i-1})$,  and $((\CA_i^H)^{H \cap H_i}, \kappa_i^*)$. Then the Euler multiplicity  $\kappa_i^*$ is just the 
		Ziegler multiplicity $\kappa$ of the restriction of $\CA_i^{H_i}$ to $H \cap H_i$. 
	\end{enumerate}
\end{lemma}

\begin{proof}
	(1).  This is immedidate from Theorem \ref{thm:zieglermulti}.
	
	(2). Fix $i \in \{1,\ldots,n\}$. Note that $(\CA_i^H)^{H \cap H_i} = (\CA_i^{H_i})^{H \cap H_i}$.  
	Let $\kappa$ be the Ziegler multiplicity of the restriction of $\CA_i^{H_i}$ with respect to $H \cap H_i$.
	We need to show that $\kappa_i^{\ast} = \kappa$.
	Therefore, for a fixed $X \in (\CA_i^H)^{H \cap H_i}$, there are $a,b \in \mathbb{N}$ such that 
	$$\exp \left((\CA_i^H)_X,(\kappa_i)_X\right) = \{a,b,0,\ldots,0\}$$ and 
	$$\exp \left((\CA_{i-1}^H)_X,(\kappa_{i-1})_X\right) = \{a,b-1,0,\ldots,0\}.$$
	By Definition \ref{def:Euler}, we have $\kappa_i^{\ast}(X) = a$. Since $\CA_i$ is free, so is $(\CA_i)_X$, and since $(\CA_i^H)_X =   ((\CA_i)_X)^H $, it follows from Theorem \ref{thm:zieglermulti} that  
	$$\exp((\CA_i)_X) = \{1,a,b,0,\ldots,0\}$$ and with the same argument for $\CA_{i-1}$ it follows that $(\CA_{i-1})_X$ is free with
	$$\exp((\CA_{i-1})_X) = \{1,a,b-1,0,\ldots,0\}.$$  
	Hence, by Theorem \ref{thm:add-del-simple}, $(\CA_i)_X^{H_i}$ is free with $\exp\left((\CA_i)_X^{H_i}\right) = \{1,a,0,\ldots,0\}$ and so thanks to Theorem \ref{thm:zieglermulti}, its multi-restriction to $H\cap H_i$ is free with 
	$\exp\left(\left((\CA_i)_X^{H_i}\right)^{H\cap H_i},\kappa_X\right) = \{a,0,\ldots,0\}$ where $\kappa_X$ is the Ziegler multiplicity. Thus, $\kappa_X(X) = a$,	since $\left((\CA_i)_X^{H_i}\right)^{H\cap H_i} = \{X\}$. Consequently,  $\kappa(X) = \kappa_X(X) = a$, as desired.
\end{proof}

\begin{remarks}
	\label{rem:free_sequence_to_ziegler_multiplicity}
	(i). In general the Euler multiplicity in connection with the restriction of a multiarrangement is not combinatorial, cf.~\cite[Ex.~4.2]{abeteraowakefield:euler}. In contrast, the canonical multiplicity $\kappa$ in a restriction of a simple arrangement is combinatorial, by definition. It just depends on the intersection lattice of the underlying arrangement. Thus it is worth pointing out that in Lemma \ref{lemma:free_sequence_to_ziegler_multiplicity}(2), the Euler multiplicity $\kappa_i^*$  here \emph{is} always combinatorial.
	
	(ii). Our proof of Lemma \ref{lemma:free_sequence_to_ziegler_multiplicity}(2) above does use the freeness of the subarrangements $\CA_i$. The following example illustrates
	that the latter is essential.	
\end{remarks}

\begin{example}
	\label{ex:zigler-vsEuler}
	Let $\CA$ be the $\BBQ$-arrangement in $\BBQ^4$ given by the defining polynomial
	\begin{equation*}
	\begin{split}
	Q(\CA) = uxyz(u+x)(u+y)(u+z)(u+x+y)(u+x+z)(y+z) \\ (u+x+y+z)(2u+x+y)(3u+x+y+z)(u+2x).
	\end{split}
	\end{equation*}
	Let $H_u := \ker u$ and $H_x := \ker x$.
	The Ziegler restriction of $\CA$ on $H_u$ is given by
	$$Q(\CA^{H_u},\kappa) = x^3y^2z^2(x+y)^2(x+z)(y+z)(x+y+z)^2$$
	and the Euler multiplicity of $(\CA^{H_u},\kappa)$ with respect to  
	to $H_u \cap H_x$ is given by 
	$$Q((\CA^{H_u})^{H_u \cap H_x},\kappa^\ast) = y^3z^3(y+z)^3,$$
	thanks to \cite[Prop.~4.1(7)]{abeteraowakefield:euler}.
	On the other hand we have
	$$Q(\CA^{H_x}) = uyz(u+y)(u+z)(y+z)(u+y+z)(2u+y)(3u+y+z)$$ 
	and here the Ziegler restriction $\kappa'$ of $\CA^{H_x}$ to $H_x \cap H_u$ is given by
	$$Q((\CA^{H_x})^{H_x\cap H_u},\kappa') = y^3z^2(y+z)^3.$$
	So in contrast to Lemma \ref{lemma:free_sequence_to_ziegler_multiplicity}(2) here the two multiplicities do not coincide.
	In this example the arrangements $\CA \setminus \{H_x\}$ and $\CA$ are not free.
	Hence the assumption on the freeness of the $\CA_i$ in 
	Lemma \ref{lemma:free_sequence_to_ziegler_multiplicity} cannot be removed.
\end{example}

\subsection{}
\label{ssec:inductive}
As in the simple case, Theorem \ref{thm:add-del} motivates 
the notion of inductive freeness. 

\begin{defn}[{\cite[Def.~0.9]{abeteraowakefield:euler}}]
	\label{def:indfree}
	The class $\CIFM$ of \emph{inductively free} multiarrangements 
	is the smallest class of multiarrangements subject to
	\begin{itemize}
		\item[(i)] $\Phi_\ell \in \CIFM$ for each $\ell \ge 0$;
		\item[(ii)] for a multiarrangement $(\CA, \mu)$, if there exists a hyperplane $H_0 \in \CA$ such that both
		$(\CA', \mu')$ and $(\CA'', \mu^*)$ belong to $\CIFM$, and $\exp (\CA'', \mu^*) \subseteq \exp (\CA', \mu')$, 
		then $(\CA, \mu)$ also belongs to $\CIFM$.
	\end{itemize}
\end{defn}

\begin{remark}
	\label{rem:rank2indfree}
	As for simple arrangements, if $r(\CA) \le 2$,
	then $(\CA, \mu)$  is inductively free,  
	\cite[Cor.~7]{ziegler:multiarrangements}.
\end{remark}

\subsection{}
\label{ssec:recursive}

Theorem \ref{thm:add-del} also motivates the concept of recursive freeness.

\begin{defn}[{\cite[Def.~2.21]{hogeroehrleschauenburg:free}}]
	\label{def:recfree}
	The class $\CRFM$ of \emph{recursively free} multiarrangements 
	is the smallest class of multiarrangements subject to
	\begin{itemize}
		\item[(i)] $\Phi_\ell \in \CRFM$ for each $\ell \ge 0$;
		\item[(ii)] for a multiarrangement $(\CA, \mu)$, if there exists a hyperplane $H_0 \in \CA$ such that both
		$(\CA', \mu')$ and $(\CA'', \mu^*)$ belong to $\CRFM$, and $\exp (\CA'', \mu^*) \subseteq \exp (\CA', \mu')$, 
		then $(\CA, \mu)$ also belongs to $\CRFM$;
		\item[(iii)] for a multiarrangement $(\CA, \mu)$, if there exists a hyperplane $H_0 \in \CA$ such that both
		$(\CA, \mu)$ and $(\CA'', \mu^*)$ belong to $\CRFM$, and $\exp (\CA'', \mu^*) \subseteq \exp (\CA, \mu)$, 
		then $(\CA', \mu')$ also belongs to $\CRFM$.
	\end{itemize}
\end{defn}

In the special instance in Definition \ref{def:recfree} when $\mu \equiv \one$,  
we recover the notion of recursive freeness from \cite[Def.~4.60]{orlikterao:arrangements} for simple arrangements.
Clearly, if $\CA$ is inductively free, it is recursively free, and likewise for multiarrangements. 

\begin{remark}
	\label{rem:recchain}
	Suppose that $(\CA, \nu) \in \CRFM$. 
	It follows from Definition \ref{def:recfree} that there exists a chain of 
	recursively free multiarrangements, starting with the 
	empty arrangement 
	\[
	\Phi_\ell \subseteq (\CA_1, \nu_1) \subseteq (\CA_2, \nu_2) \ldots (\CA_n, \nu_n) = (\CA, \nu) 
	\]
	such that each consecutive pair obeys Definition \ref{def:recfree}. 
	In particular, $|\nu_i| = |\nu_{i-1}| \pm 1$ for each $1 \le i \le n$ and 
	$|\nu| = n$.
	We refer to a sequence as above as 
	a \emph{recursive chain} of $(\CA, \nu)$. 
\end{remark}

\section{Filtrations of Free Multiplicities}
\label{sec:filtrations}
\subsection{}
Several of the results in this section are taken from \cite[\S3]{hogeroehrle:Ziegler}.
We begin with recalling a natural partial order on the 
set of multiplicities for a simple arrangement.

\begin{defn}
	\label{def:order}
	For multiplicities $\mu_1$ and $\mu_2$ on a fixed arrangement $\CA$, define 
	$\mu_1 \le \mu_2$ provided $\mu_1(H) \le \mu_2(H)$ for every $H$ in $\CA$.
\end{defn}

\begin{defn}
\label{def:freefiltration}
Let $\CA$ be a free arrangement. 
Suppose there is a free multiplicity $\mu > \one$ on $\CA$ 
such that 
there is a sequence of free multiplicities $\mu_i$ on $\CA$ satisfying
$\mu_i < \mu_{i+1}$ and $|\mu_{i+1}| = |\mu_i| +1$,
for $i = 1, \ldots, n-1$, where $\mu_1 := {\one}$ and $\mu_n := \mu$.
Then we say that the sequence of multiarrangements 
$(\CA, \mu_i)$ is a \emph{free filtration} of $(\CA, \mu)$ or
simply that the sequence 
$\mu_i$ is a \emph{filtration} of free multiplicities on $\CA$.
\end{defn}

It is natural to consider the 
generalization of Definition \ref{def:addfree} for 
multiarrangements.

\begin{defn}
	\label{def:multaddfree}
	The multiarrangement $(\CA, \mu)$ is said to be  
	\emph{additively free} if there is a free filtration of multiplicities $\mu_i$ on $\CA$,
	\[
	\mu_0 < \mu_1 < \cdots < \mu_n = \mu,
	\]
	i.e., where each $(\CA, \mu_i)$ is free with $|\mu_i| = i$.
	Denote this class by $\CAFM$. In particular, $(\CA, \mu_0) = \Phi_\ell$. 
	For $\mu \equiv \one$, this 
	specializes to  Definition \ref{def:addfree}.
\end{defn}

\begin{remarks}
	\label{rem:multaddfree}
	(i). It is obvious from the definitions that if $(\CA, \mu)$ is inductively free, it is 
	additively free. The converse is false already for simple arrangements, see \cite{hogeroehrle:stairfree}.
	
	(ii). If $(\CA, \mu)$ is additively free, it need not be the case that $\CA$ itself is additively free, see Example \ref{ex:indfree-recfree}.
\end{remarks}

For an arrangement $\CA$, we choose $\alpha_H \in S$ so that $H = \ker \alpha_H$,  for each $H \in \CA$.
For a multiplicity $\mu$ on $\CA$, we associate the following canonical derivation 
in $D(\CA,\mu)$,
\begin{equation}
\label{eq:thetamu}
\theta_\mu := \left(\prod_{H \in \CA} \alpha_H^{\mu(H)-1}\right) \theta_E,
\end{equation}
where 
$\theta_E$ is the Euler derivation in $\Der(S)$, 
\cite[Def.~4.7]{orlikterao:arrangements}. 
We consider the case when $\theta_\mu$ 
belongs to a basis of $D(\CA,\mu)$.

\begin{lemma}
	\label{lem:thetamu}
	Suppose $(\CA,\mu)$ is free and $\theta_\mu$
	is part of a basis of $D(\CA,\mu)$. Then $(\CA,\nu)$ is free for every multiplicity $\nu$ with $\one \le \nu \le \mu$. Moreover, if $\exp (\CA) = \{1,e_2,\ldots,e_\ell\}$, then  $\exp (\CA,\nu) = \{1+ \vert \nu\vert - \vert \CA \vert,e_2,\ldots,e_\ell\}$ for each such $\nu$.
\end{lemma}

\begin{proof}
	Suppose $\{\theta_\mu, \theta_2, \ldots, \theta_\ell\}$ is a basis of $D(\CA,\mu)$.
	Then the first statement is immediate by 
	Theorem \ref{thm:ziegler-saito} with $\theta_\mu$ replaced by $\theta_\nu$.  Moreover, if  $\theta_2, \ldots, \theta_\ell$ are homogeneous, the second statement follows. 
\end{proof}

\begin{corollary}
	\label{cor:thetamu}
	Suppose $(\CA,\mu)$ is free and that $\pdeg \theta_\mu = \min \exp(\CA,\mu)$. Then $(\CA,\nu)$ is free for every multiplicity $\nu$ with $\one \le \nu \le \mu$, and   $\exp (\CA,\nu) = \{1+ \vert \nu\vert - \vert \CA \vert,e_2,\ldots,e_\ell\}$, where $\exp (\CA) = \{1,e_2,\ldots,e_\ell\}$.
\end{corollary}

\begin{proof}
	Since $\pdeg \theta_\mu = \min \exp(\CA,\mu)$, $\theta_\mu$ can be chosen as part of a basis of  $D(\CA,\mu)$. The result thus follows from Lemma \ref{lem:thetamu}.
\end{proof}

\begin{lemma}\label{free_deletion_of_multiarrangement}
	Assume $(\CA,\mu)$ is free and there is an $H_0 \in \CA$ with $\mu(H_0)=1$ such that
	the corresponding deletion $(\CA',\mu')$ is also free.
	Suppose one of the following 
	holds:
	\begin{enumerate} 
		\item [(i)] $\min \exp(\CA,\mu) = \vert \mu\vert - \vert \CA' \vert$.
		\item [(ii)] $\theta_\mu$ is part of a basis of $D(\CA,\mu)$ and $\vert \CA' \vert \not= \vert \mu^\ast \vert$.
	\end{enumerate}
	Then $(\CA',\nu)$ is free for every multiplicity $\nu$ with $\one \le \nu \le \mu'$. In particular, $\CA'$ is free.
\end{lemma}

\begin{proof}
	According to Theorem \ref{thm:restriction} there exists a basis $\{\theta_1, \ldots, \theta_\ell \}$ of $(\CA',\mu')$ such that
	$B : = \{\theta_1, \ldots, \alpha_0 \theta_k, \ldots, \theta_\ell \}$
	is a basis of $D(\CA, \mu)$.
	We claim that conditions (i) and (ii) ensure that we may pick $\theta_\mu$ as one of the $\theta_1,\ldots,\theta_{k-1},\theta_{k+1},\ldots,\theta_\ell$.
	Assuming the claim, $\theta_\mu$ is also part of a basis of $D(\CA',\mu')$ and thus the result follows from Lemma \ref{lem:thetamu} applied to $(\CA',\mu')$.
	
	If condition (i) is satisfied then $\theta_\mu$ is a derivation of $(\CA,\mu)$ of minimal degree, by Corollary \ref{cor:thetamu}. 
	Since $\theta_\mu$ is not divisible by $\alpha_0$, it can be chosen as a part of the basis $B$.
	
	In view of Theorems \ref{thm:restriction} and \ref{thm:add-del},  condition (ii) implies that 
	\begin{equation}
	\label{eq:conditionb}
		\pdeg(\theta_\mu) = \vert \mu \vert - \vert \CA' \vert \not= \vert \mu \vert - \vert \mu^\ast \vert = \pdeg (\alpha_0 \theta_k). 
	\end{equation}
	The derivations of a basis of $D(\CA,\mu)$ do not depend on the picked derivations of other degrees. 
	Hence $\theta_\mu$ can be chosen as a part of the basis $B$.
	The claim follows.
\end{proof}

\begin{remark}
	Equation \eqref{eq:conditionb} shows that the second part in condition (ii) of Lemma \ref{free_deletion_of_multiarrangement} is necessary, since else $\alpha_0 \theta_k$ might be an $S$-linear combination of $\theta_\mu$ and a derivation of lower degree.
\end{remark}

We illustrate Lemma \ref{free_deletion_of_multiarrangement}  with two easy examples. In both cases, both conditions of the lemma fail, though the final statement in the lemma still holds in the first instance.  

\begin{example}
	Define $(\CA,\mu)$ by
	$$
	Q(\CA,\mu) = x^2y^2z(x+y)(y+z).
	$$
	Then $\theta_1 := x^2D_x - y^2D_y + z^2 D_z$, $\theta_2 := z(y+z)D_z$, and $\theta_\mu = xy \theta_E$ form a basis of $D(\CA,\mu)$.
	Let $H_0 := \ker(x+y)$. A basis  of $D(\CA,\mu)$ as in Theorem \ref{thm:restriction} is given by $\theta_1$, $\theta_2$, and $\theta_3 := (x+y)( y^2 D_y - z^2 D_z)$.
	Hence $\exp(\CA',\mu') = \{2,2,2\}$ and so  $\vert \CA' \vert = 4 = \vert \mu^\ast \vert$, by Theorem \ref{thm:add-del}.
	Also, $\min \exp(\CA,\mu) = 2 < \vert \mu\vert - \vert \CA' \vert = 7 - 4 = 3$.
	Thus neither condition in Lemma \ref{free_deletion_of_multiarrangement} is satisfied. Nevertheless, one can check that the statement of the conclusion of the lemma still holds: 
	$(\CA',\nu)$ is free for every multiplicity $\nu$ with $\one \le \nu \le \mu'$.
\end{example}

In our second example the conclusion of Lemma \ref{free_deletion_of_multiarrangement} fails since $\CA'$ is not free.

\begin{example}
	Define $(\CA,\mu)$ by
	$$
	Q(\CA,\mu) = x^2y^2z^2(x+y)(x-y)(x-z)(y+z).
	$$
	Then $\theta_1 := x^3D_x + y^3D_y + z^3 D_z$, $\theta_2 := x^2yD_x+xy^2D_y +z^2(y+z-x)D_z$, $\theta_\mu = xzy \theta_E$ form a basis of $D(\CA,\mu)$.
	Let $H_0 := \ker(x+y)$. A basis as in Theorem \ref{thm:restriction} of $D(\CA,\mu)$ is given by $\theta_1$, $\theta_2$, and $\theta_3 := (x+y)( x^2z D_x + y^2(y+z-x)D_y + xz^2 D_z)$.
	Hence $\exp(\CA',\mu') = \{3,3,3\}$ and so  $\vert \CA' \vert = 6 = \vert \mu^\ast \vert$, by Theorem \ref{thm:add-del}.
	Also, $\min \exp(\CA,\mu) = 3 < \vert \mu\vert - \vert \CA' \vert = 10 - 6 = 4$.
	Thus neither condition in Lemma \ref{free_deletion_of_multiarrangement} is satisfied.
	Now $\CA'$ is not free, since $\exp(\CA) = \{1,3,3\}$ and $\vert \CA'' \vert = 3 \not=1+3$. So in this instance, the conclusion of Lemma \ref{free_deletion_of_multiarrangement} is false.
\end{example}

\begin{lemma}
	[{\cite[Lem.~3.2]{hogeroehrle:Ziegler}}]
\label{lem:indfreechain}
Let $\CA$ be an inductively free arrangement. 
Suppose there is a free multiplicity $\mu > \one$ on $\CA$ 
along with a free filtration  
$\one = :\mu_1 < \ldots < \mu_n := \mu$ 
of $(\CA, \mu)$.
If each restriction  
along the free filtration is inductively free, then 
so is $(\CA, \mu)$.
\end{lemma}

In our next result we present a mild condition on a free multiplicity $\mu$ 
of a free arrangement $\CA$ which implies that every intermediate 
multiplicity $\one < \nu < \mu$ is also free.

\begin{lemma} 
		[{\cite[Lem.~3.4]{hogeroehrle:Ziegler}}]
\label{lem:multifreetube}
Let $\CA$ be a free arrangement with exponents $1 \le e_2 \le \ldots \le e_\ell$.
Assume that there is a free multiplicity $\mu > \one$ on $\CA$ with
$\exp(\CA,\mu) = \{e, e_2,\ldots,e_\ell\}$, where   
$e = 1 + \vert \mu \vert - \vert \CA \vert$.
Suppose that $\vert \mu \vert - \vert \CA \vert \ge e_\ell$. 
Let $\nu$ be a multiplicity satisfying $\one < \nu <  \mu$.
Then $(\CA,\nu)$ is free with 
$\exp(\CA,\nu) = \{\tilde e , e_2, \ldots, e_\ell\}$, where
$\tilde e = 1+ \vert \nu \vert - \vert \CA \vert$.
\end{lemma}

We record an important consequence of Lemma \ref{lem:multifreetube}
and Theorem \ref{thm:add-del} which shows that 
the multisets of exponents of the restrictions 
along a free filtration, as in Lemma \ref{lem:multifreetube}, 
are constant and do not depend on 
the Euler multiplicities $\mu^*_i$.

\begin{corollary}
		[{\cite[Cor.~3.8]{hogeroehrle:Ziegler}}]
\label{cor:multifreetuberestrictions}
Let $\CA$ and $(\CA, \mu)$ be as in Lemma \ref{lem:multifreetube}.
Let $\one = :\mu_1 < \ldots < \mu_n := \mu$ be 
a free filtration of $(\CA, \mu)$. 
Then for each restriction along the chain, we have $(\CA'', \mu^*_i) = (\CA'',\kappa)$,
where $\kappa$ is Ziegler's canonical multiplicity on $\CA''$.
In particular, each such restriction along the filtration is free with 
$\exp (\CA'', \mu^*_i) = \exp(\CA'',\kappa) = \{e_2, \ldots, e_\ell\}$, 
where $\exp \CA =\{1, e_2, \ldots, e_\ell\}$.
\end{corollary}

\subsection{}
The next lemma is a crucial ingredient in our proofs in \S \ref{Sec:proofs} of Theorems \ref{thm:main} to \ref{thm:main6}.  

	\begin{lemma}
		\label{lemma:free_sequence_to_simple_multiplicity}
		Let $\CA$ be an $\ell$-arrangement. Suppose $H_0 \in \CA$ such that $\CA' = \CA \setminus \{H_0\}$ and $\CA'' = \CA^{H_0}$ are free.
		Let $(\CA'',\kappa)$ be the Ziegler restriction of $\CA$ to $H_0$. Then $D(\CA'',\mu)$ is free for every multiplicity
		$\one \le \mu \le \kappa$ with $\exp(\CA'',\mu) = \{1+|\mu| - |\CA''|, e_2, \ldots, e_{\ell-1}\}$, where $\exp \CA'' = \{1, e_2, \ldots, e_{\ell-1}\}$.
		
		Moreover, let $(\CA'',\mu_i)$ be a free filtration of $(\CA'',\kappa)$, as in Definition \ref{def:freefiltration}, and let $H_i$ be the hyperplane of $\CA''$ involved in the $i$-th addition step in this filtration.
		Then the Euler restriction $\mu_i^*$ is the Ziegler restriction of $\CA''$ to $H_i \cap H_0$.
	\end{lemma}

   \begin{proof}
   	Owing to \cite[Thm.~1.13]{abe:projectivedimension}, 
	the restriction map $\rho \colon D(\CA) \rightarrow D(\CA'')$ from \eqref{eq:rho}
	is surjective, since $\CA'$ is free.
	Let $\alpha_0 \in V^*$ with $H_0 = \ker \alpha_0$ and let 
	\[
	D_{H_0}(\CA) := \{\theta \in D(\CA) \mid \theta(\alpha_0) = 0\}
	\]
	be the annihilator of $H_0$ in $D(\CA)$ which is a
	graded $S$-submodule of $D(\CA)$. 
	Then, since 
	\[
	D(\CA) = S \theta_E \oplus D_{H_0}(\CA),
	\]
	as $S$-modules (cf.~\cite[Prop.~4.27]{orlikterao:arrangements}), 
	$\CA''$ is free, and $\rho$ is surjective, 
	we may pick an $\overline{S}$-basis of homogeneous derivations $\theta_E,\theta_2,\ldots,\theta_{\ell-1}$ of $D(\CA'')$ such 
	that $\theta_2,\ldots,\theta_{\ell-1} \in \rho(D_{H_0}(\CA))$. 
	We observe that the restriction of $\rho$ to $D_{H_0}(\CA)$ affords a restriction map 
	\[
	\rho|_{D_{H_0}(\CA)} : D_{H_0}(\CA) \to  D(\CA'',\kappa),
	\]
	\cite{ziegler:multiarrangements}; see also \cite[Thm.~1.34(i)]{yoshinaga:free14}.
	
	Hence $\theta_2,\ldots,\theta_{\ell-1} \in D(\CA'',\kappa)$.	
	    Let $H = \ker \alpha_H$ for $H \in \CA''$ and recall from \eqref{eq:thetamu} that $\theta_\kappa:= \prod_{H \in \CA''} \alpha_H^{\kappa(H)-1} \cdot \theta_E \in D(\CA'',\kappa)$.

	Since the chosen basis elements $\theta_E,\theta_2,\ldots,\theta_{\ell-1}$ of $D(\CA'')$ are independent over $\overline{S}$,	it follows that also 
	\begin{equation*}
	\label{eq:basis1}
	\theta_\kappa, \theta_2,\ldots,\theta_{\ell-1} \in D(\CA'',\kappa)
	\end{equation*}
	are independent over $\overline{S}$	and thus 
	\[
	\det M(\theta_\kappa,\ldots,\ldots \theta_{\ell-1}) 
	= \prod_{H\in \CA''} \alpha_H^{\kappa(H)-1} \det M (\theta_E,\ldots,\theta_{\ell-1}) \ \dot=   \prod_{H\in \CA''} \alpha_H^{\kappa(H)}
	\]
	is a non-zero scalar multiple of $Q(\CA'',\kappa) $.    
	Hence $(\CA'',\kappa)$ is free, thanks to Theorem \ref{thm:ziegler-saito}.
		
	Likewise and more generally, $(\CA'',\mu)$ is free for every multiplicity
	$\one \le \mu \le \kappa$; explicitly, 
	\begin{equation}
	\label{eq:basis}
	\left\{ \theta_\mu , \theta_2, \ldots, \theta_{\ell-1}\right\}
	\end{equation}
	is an $\overline{S}$-basis of $D(\CA'',\mu)$, by the same argument, as above.
	
	The fact that the corresponding Euler multiplicities $\mu_i^*$ in a free filtration of $(\CA'',\kappa)$
	are Ziegler multiplicities on $\CA''$ follows from a concentrated multiplicity on $\CA''$; see the proof of Corollary \ref{cor:multifreetuberestrictions} in 
	\cite[Cor.~3.8]{hogeroehrle:Ziegler}.
	
	Finally, the statement about the exponents of $(\CA'',\mu)$ follows readily from the proof above and \eqref{eq:basis}. 
\end{proof}

\begin{remark}
	\label{rem:euler-being-kappa}
	As already pointed out in Remark \ref{rem:free_sequence_to_ziegler_multiplicity}(i), in general the Euler multiplicity in connection with the restriction of a multiarrangement is not combinatorial, cf.~\cite[Ex.~4.2]{abeteraowakefield:euler}. In contrast, a canonical multiplicity in a restriction of a simple arrangement is combinatorial, by definition. 
	Again, as in the case of Lemma \ref{lemma:free_sequence_to_ziegler_multiplicity}(2), 
	likewise in the situations of Lemma \ref{lemma:free_sequence_to_simple_multiplicity}
	and Corollary \ref{cor:multifreetuberestrictions},
	the Euler multiplicities $\mu_i^*$ are always combinatorial. 
\end{remark}

\section{Proofs of Theorems \ref{thm:main} --- \ref{thm:main6}}
\label{Sec:proofs}

\subsection{}

Armed with the results from the previous sections, we address the main theorems from the introduction starting with the first three of these.
	
	\begin{proof}[Proof of Theorem \ref{thm:main}]
		We argue by induction on $\lvert \CA \rvert$. 
		For $\lvert \CA \rvert \le 3$ the assertion is true since all restrictions are of rank at most $2$ and all multiplicities on such are inductively free, \cite{ziegler:multiarrangements}. 
		Now suppose $\lvert \CA \rvert > 3$. Since $\CA$ is inductively free, there is an $H \in \CA$ such 
		that both $\CA' := \CA \setminus \{H\}$ and $\CA'' := \CA^{H}$ are inductively free. We consider two cases.
		
		Suppose $H = H_0$. It then follows from Lemma \ref{lemma:free_sequence_to_simple_multiplicity} that there is a sequence of  
		free multiplicities from the inductively free (multi-)arrangement $\CA'' = (\CA'',\one)$ to $(\CA'',\kappa)$
		and each occurring Euler multiplicity in the sequence is a Ziegler multiplicity of a restriction of $\CA''$. Now each of the latter is inductively free by induction, since $\lvert \CA'' \rvert < \lvert \CA \rvert$ and $\CA''$ is inductively free. Finally, $(\CA'',\kappa)$ is inductively free, thanks to Lemma \ref{lem:indfreechain}.
		
		Suppose that $H \ne  H_0$. Now we use Lemma \ref{lemma:free_sequence_to_ziegler_multiplicity} and consider 
		the triple $(\CA^{H_0},\kappa)$, $\left((\CA')^{H_0},\kappa'\right)$,  and $\left((\CA'')^{H \cap H_0}, \kappa^*\right)$. Then the Euler multiplicity  $\kappa^*$ is just the 
		Ziegler multiplicity, say $\kappa''$, of the restriction of $\CA''$ to $H \cap H_0$. Therefore, as 
		$\lvert \CA'' \rvert < \lvert \CA \rvert$ and $\lvert \CA' \rvert < \lvert \CA \rvert$ and both 
		$\CA''$ and $\CA'$ are inductively free, it follows by induction that both $\left((\CA')^{H_0},\kappa'\right)$ and $\left((\CA'')^{H \cap H_0}, \kappa^*\right) = \left((\CA'')^{H \cap H_0}, \kappa''\right)$ are inductively free. It follows from 
		Theorem \ref{thm:restriction} 
		that the exponents of $(\CA^{H_0},\kappa)$ and $\left((\CA')^{H_0},\kappa'\right)$
		differ in precisely one entry by $1$.  
		So that, by Theorem \ref{thm:add-del}, 
		$\left((\CA'')^{H \cap H_0}, \kappa^*\right)$ is free and the exponents satisfy 
		$\exp\left((\CA'')^{H \cap H_0}, \kappa^*\right) \subseteq \exp \left((\CA')^{H_0},\kappa'\right)$.
		It thus follows from Definition \ref{def:indfree} that also $(\CA^{H_0},\kappa)$ is inductively free.
	\end{proof}

\begin{proof}[Proof of Theorem \ref{thm:main5}]
	Let $\CA = (\CA,V)$.		We argue by induction on $\dim V$. 
	If $\dim V \le 3$, then the assertion holds, since all restrictions are of rank at most $2$ and all multiplicities on such are inductively free and so in particular, are recursively free. 
	
	Now suppose $\dim V > 3$ and that the statement holds for recursively free arrangements in smaller dimensions. 
	
	Since $\CA$ is recursively free, it admits a recursive chain, as in Remark \ref{rem:recchain}.
	We argue further  by induction on the length of such a recursive chain. Obviously, if the length of such a chain is at most three, then $\CA$ is inductively free, and then so is $(\CA'',\kappa)$, by Theorem \ref{thm:main}, and so $(\CA'',\kappa)$ is recursively free.
	
	Now fix a recursive  chain for $\CA$ and assume that its length is at least four and that the result holds for  recursively free arrangements admitting a shorter recursive chain than the fixed one for $\CA$.  We consider two cases.
	
	Suppose $H_0$ is the last hyperplane in the fixed chain for $\CA$. It follows from Lemma \ref{lemma:free_sequence_to_simple_multiplicity} that
	$(\CA'',\mu)$ is free for every multiplicity $\mu$ satisfying $\one \le \mu \le \kappa$ and furthermore, every corresponding Euler multiplicity $\mu^*$ is a Ziegler multiplicity
	of a restriction of $\CA''$. Since $\CA'' = (\CA'', \one)$ is recursively free and $\dim H_0 < \dim V$, it follows by induction on dimension  that these Ziegler restrictions are recursively free.
	Fix a chain of free multiplicities $\one \le \mu_i \le \kappa$, as in Definition \ref{def:freefiltration}. Arguing by induction on $i$, we have that   
	$(\CA'', \mu_i')$ is recursively free. It follows from Theorems \ref{thm:restriction} and \ref{thm:restriction2} that the exponents of the corresponding restriction form a subset of $\exp(\CA'', \mu_i')$. Thus 
	$(\CA'', \mu_i)$ is recursively free, by Definition \ref{def:recfree}.
	By induction on $i$, we get that $(\CA'',\kappa)$ is recursively free, as desired.
	(Note that the argument shows that indeed $(\CA'',\mu)$ is recursively free for every multiplicity $\mu$ satisfying $\one \le \mu \le \kappa$.)

	Let $H \not= H_0$ be the last hyperplane in our recursive chain for $\CA$. 
	Let $\kappa_1$ be the Ziegler multiplicity of $\CB := \CA \setminus \{H\}$ restricted to $H_0$. 
	By Lemma \ref{lemma:free_sequence_to_ziegler_multiplicity}, both $(\CA'',\kappa)$ and $(\CB'',\kappa_1)$ are free and the multiplicities $\kappa$ and $\kappa_1$ 
	differ only in one hyperplane, namely $\kappa(H \cap H_0) = \kappa_1(H \cap H_0) +1 $. By induction on the length of the recursive chain, one of $(\CA'',\kappa)$ or $(\CB'',\kappa_1)$
	is recursively free (depending on whether $H$ has been deleted or added in the last step in the chain, respectively). Hence we only have to check that the corresponding Euler restriction 
	$\left((\CA'')^{H \cap H_0},\kappa^*\right)$ is recursively free. This in turn holds by Lemma \ref{lemma:free_sequence_to_ziegler_multiplicity}, 
	since this is the Ziegler restriction, say $\kappa''$ on $(\CA^{H})^{H \cap H_0}$. Since 
	 $\dim H < \dim V$ and $\CA^H$ 
	is recursively free, $\left((\CA'')^{H \cap H_0},\kappa^*\right) = \left((\CA^{H})^{H \cap H_0},\kappa''\right)$ is recursively free, by induction on dimension. 	
\end{proof}

We obtain Theorem \ref{thm:main2} as a consequence of Theorem \ref{thm:main} (resp.~Theorem \ref{thm:main5})
and Lemmas \ref{lem:indfreechain} and  \ref{lemma:free_sequence_to_simple_multiplicity}.	

\begin{proof}[Proof of Theorem \ref{thm:main2}]
	Let $\CA'' = \CA^{H_0}$.     By Lemma \ref{lemma:free_sequence_to_simple_multiplicity}, $(\CA'',\mu)$ is free for each $\one \le \mu \le \kappa$. Therefore, 
	$(\CA'',\kappa)$ is additively free whenever $\CA''$ is so.
	
	For the other two properties we have to check that the corresponding
	Euler restriction corresponding to the multiplicity $\mu^*$ has also the same property.
	First suppose $\CA''$ is inductively free. It follows from Lemma \ref{lemma:free_sequence_to_simple_multiplicity} that there is a sequence of  
	free multiplicities from the inductively free (multi-)arrangement $\CA'' = (\CA'',\one)$ to $(\CA'',\kappa)$
	and each occurring Euler multiplicity in the sequence is a Ziegler multiplicity of a restriction of $\CA''$.  It thus follows from Theorem \ref{thm:main}, applied to $\CA''$, 
	that each of the latter is inductively free.
	Finally, $(\CA'',\kappa)$ is inductively free, thanks to Lemma \ref{lem:indfreechain}.
	
	The argument in the previous paragraph also applies in case when $\CA''$ is recursively free where in place of Lemma \ref{lem:indfreechain} we require an analogue for recursive freeness and Theorem \ref{thm:main5} is used in place of  Theorem \ref{thm:main}.
\end{proof}		

\subsection{}
Theorem \ref{thm:main6} is a consequence of the following example of an additively free arrangement which admits a  Ziegler restriction that is not additively free. 

	\begin{example}
		\label{ex:additive}
		Consider the rank $5$ arrangement $\CD$ over the reals whose defining polynomial is the product of the $21$ linear forms given 		by 
		    \begin{align*}
		Q(\CD) = \ &x_2(x_1+x_3-x_5)(2x_1+x_2+x_3)(2x_1+x_2+2x_3+x_4-x_5)\\
		&x_5(x_1+x_3)(x_2+x_5)(2x_1+x_2+2x_3+x_4)(2x_1+x_3-x_5)\\
		&(2x_1+2x_2+2x_3+x_4)(x_2+x_3+x_4)(x_1+x_2+x_3+x_4)\\
		&(x_3+x_4)(x_1+x_2+x_3)x_1(x_1+x_3+x_4)(2x_1+x_2+x_3-x_5)\\
		&(x_2+x_3+x_4+x_5)(x_1-x_5)(x_1-x_4-x_5)x_4.
		\end{align*}	
		
		This arrangement occurs naturally as a subarrangement of a particular rank $5$ restriction of the Weyl arrangement of type $E_7$, see \cite[\S 3.3]{hogeroehrle:stairfree}.  
		It has a remarkable rare property: it is additively free but not inductively free.
		A free filtration of $\CD$ 
		is given 		by the order of the hyperplanes as they appear above in $Q(\CD)$
		ending in $\ker (x_4)$, cf.~\cite[Table 4]{hogeroehrle:stairfree}.
		We aim to show that the Ziegler restriction $(\CD^{\ker(x_4)},\kappa)$ is not additively free, thus proving Theorem~\ref{thm:main6}.
		
		Abbreviate $\CB:= \CD^{\ker(x_4)}$. 
By \cite[\S 3.3]{hogeroehrle:stairfree}, $\CD$ is additively free with $\exp(\CD) = \{1,5,5,5,5\}$, $\CB$ is free with $\exp(\CB) = \{1,5,5,5\}$, but $\CB$ is not additively  free.
One easily checks from $Q(\CD)$ above that 
		$\CB$ consists of $16$ hyperplanes and that $(\CB,\kappa)$ is given by
		\begin{align*}
		Q(\CB,\kappa) = & \ x_1x_2x_3x_5(x_1+x_3)^2(x_1-x_5)^2(x_2+x_3)(x_2+x_5)(x_1+x_2+x_3)^3\cdot  \\& \ (x_1+x_3-x_5)(2x_1+x_3-x_5)(2x_1+x_2+x_3)(2x_1+x_2+2x_3)\cdot  \\& \ (x_2+x_3+x_5)(2x_1+x_2+x_3-x_5)(2x_1+x_2+2x_3-x_5). 		
		\end{align*}
		
	Since $\CD$ is free with with $\exp(\CD) = \{1,5,5,5,5\}$, 
$(\CB,\kappa)$ is free with $\exp(\CB,\kappa) = \{5,5,5,5\}$, by Theorem \ref{thm:zieglermulti}. 
Recall  the derivation 
$\theta_\kappa = \left(\prod_{H \in \CB} \alpha_H^{\kappa(H)-1}\right) \theta_E$ in $D(\CB,\kappa)$ from \eqref{eq:thetamu}.
Since $\min \exp(\CB,\kappa) = 5 = 1+\vert \kappa \vert - \vert \CB \vert = \pdeg(\theta_\kappa)$, $\theta_\kappa$ 
is a derivation in $D(\CB,\kappa)$ of minimal degree, and so it  
can be chosen as an element of an $S$-basis of $D(\CB,\kappa)$. Thus, thanks to Corollary \ref{cor:thetamu},  
$D(\CB,\mu)$ is free for every $\mu$ with $\one \le \mu \le \kappa$ and $\exp(\CB,\mu) = \{1+\vert \mu \vert - \vert \CB \vert,5,5,5\}$.

Now fix such a multiplicity $\mu$ on $\CB$ 
 and some $H_0 = \ker \alpha_0 \in \CB$ with $\mu(H_0) = 1$. 
We claim that the deletion $(\CB',\mu')$ is not free (i.e.~where $H_0$ is removed from $\CB$  in passing from $(\CB,\mu)$ to $(\CB',\mu')$). 
By way of contradiction, suppose $(\CB',\mu')$ is free.

Since $1+\vert \mu \vert - \vert \CB \vert \le 1+\vert \kappa \vert - \vert \CB \vert = 5$, we have $\min \exp(\CB,\mu) = 1+\vert \mu \vert - \vert \CB \vert = \vert \mu \vert - \vert \CB' \vert$.
It follows from Lemma \ref{free_deletion_of_multiarrangement}(i) that 
$\CB'$ is also free with exponents $\{1,4,5,5\}$. This however is false. For, if $\CB'$  were free, so would be $\CB''$. But we have already observed  in \cite[\S 3.3]{hogeroehrle:Ziegler} that there is no restriction $\CB''$ of $\CB$ at all which is free with exponents $\{1,5,5\}$.
		As a result, $(\CB',\mu')$ is not free after all, as claimed. Therefore, $\one$ is the smallest free multiplicity on 
		$\CB$ that we can reach by a sequence of free deletions from $(\CB,\kappa)$. And since $\CB$ itself is not additively free, thanks to \cite[\S 3.3]{hogeroehrle:Ziegler}, neither is $(\CB,\kappa)$. 	
	\end{example}

Next, we present an instance of the failure of Theorem \ref{thm:main2}, where 
the ambient arrangement is still free and only the deletion is not.

\begin{example}
	\label{ex:failure1.5}
	Consider  the triple $(\CB, \CB', \CB'')$ for the 
	arrangement $\CB$  from Example \ref{ex:additive} with respect to $H_0 = \ker(x_1+x_2+x_3)$. One checks that $\CB'$ is not free while  $\CB''$ is inductively free with exponents $\{1,3,3\}$.	
	Since $\CB$ is free with exponents $\{1,5,5,5\}$, by Example \ref{ex:additive}, $(\CB'',\kappa)$ 	is free with exponents $\{5,5,5\}$, by Theorem \ref{thm:zieglermulti}. However, the latter is not inductively free, because for any hyperplane in $\CB''$, 
	the order of the Euler multiplicity $\vert \kappa^\ast \vert$ is bounded above by $8$, and so the condition on the exponents in 
	Definition \ref{def:indfree}(ii) is not satisfied for any member of $\CB''$, as then $\vert \kappa^\ast \vert$ would have to be $10 = 5 +5$, cf.~Theorem \ref{thm:add-del}.
\end{example}

\section{Complements}

In this section we present some complements to our main developments.

\subsection{A converse to Theorem \ref{thm:main}}
\label{Yconverse}
There is an important and  celebrated converse to Ziegler's Theorem \ref{thm:zieglermulti}, due to Yoshinaga; see \cite{yoshinaga:free04} and \cite[Thm.~1.38]{yoshinaga:free14}. It is rather natural to ponder whether there might be  
a converse to Theorem \ref{thm:main} analogous of Yoshinaga's theorem:
Suppose that $\CA$ is an $\ell$-arrangement for $\ell \ge 4$, $(\CA^{H_0},\kappa)$ is inductively free 
for some $H_0 \in \CA$ 
and $\CA$ is locally inductively free along $H_0$, i.e.~$\CA_X$ is inductively free for each $X \in L(\CA)$ with $\{0\} \ne X \subset H_0$. Then one might hope that $\CA$ itself is inductively free. 

The rank $4$ reflection arrangement $\CA = \CA(G_{31})$ of the complex reflection group of type $G_{31}$ is an example which dashes this hope. Owing to
\cite[Thm.~1.4]{hogeroehrle:Ziegler}, $(\CA^{H_0},\kappa)$ is inductively free for any $H_0 \in \CA$. 
Moreover, inspecting \cite[Table C.12]{orlikterao:arrangements} and observing \cite[Cor.~6.28(2)]{orlikterao:arrangements}, it follows from \cite[Thm.~1.2]{hogeroehrle:super} that all proper localizations $\CA_X$ are supersolvable. In particular, 
$\CA$ is locally inductively free along any $H_0$.
Nevertheless, $\CA(G_{31})$ itself is not inductively free, according to 
\cite[Thm.~1.1]{hogeroehrle:inductivelyfree}.

\subsection{Divisional freeness}

It follows from Theorem \ref{thm:main6} that additive freeness of $\CA$ does not imply inductive freeness of  $(\CA'',\kappa)$ in general. 
In this section we investigate whether the weaker freeness conditions 
of divisional freeness 
on $\CA$ still entails inductive freeness of $(\CA'',\kappa)$.
Divisional freeness is a canonical notion weaker than inductive freeness due to Abe, \cite{abe:divfree}. It is defined as follows, where $\chi(\CA, t)$ denotes the characteristic polynomial of the arrangement $\CA$.

\begin{defn}
	[{\cite[Def.~1.5]{abe:divfree}}]
	\label{def:divfree}
	An $\ell$-arrangement $\CA$ is 
	\emph{divisionally free} if either $\ell \le 2$, $\CA = \Phi_\ell$, or
	else there is a flag of subspaces $X_i$ of rank $i$ in $L(\CA)$,
	\[
	X_0 = V \supset X_1 \supset X_2 \supset \cdots \supset X_{\ell -2},
	\]
	so that $\chi(\CA^{X_i},t)$ divides $\chi(\CA^{X_{i-1}},t)$, for 
	$i = 1, \ldots, \ell-2$.
	Denote this class by $\CDF$.
\end{defn}

Owing to \cite[Thm.~1.1]{abe:divfree}, each member of $\CDF$ is free.
Indeed, in \cite[Thms.~1.3 and ~1.6]{abe:divfree}, Abe observed that $\CIF \subsetneq \CDF$. For, the reflection arrangement of the complex reflection group of type $G_{31}$ is 
divisionally free but not inductively free. The latter is  not even additively free, see the proof of \cite[Lem.\ 3.5]{hogeroehrle:inductivelyfree}.

It turns out that divisonal freeness of $\CA$ does not imply  inductive freeness of $(\CA'',\kappa)$, as our next example demonstrates.

\begin{example}
	\label{ex:divfree}
		Recall the intermediate arrangements $\CA^k_\ell(r)$ from Example \ref{ex:akl}.
			Owing to \cite[Thm.~5.6]{abe:divfree}, $\CA^k_\ell(r)$ is divisionally free for $k \ge 1$.
	It follows from 
	\cite[Prop.~4.2]{hogeroehrleschauenburg:free} that 
	for $\CA = \CA^k_\ell(r)$ with $r \ge 3$, $\ell \ge 5$, and $1 \le k \le  \ell -3$, the Ziegler restriction 
	$(\CA^{H_0},\kappa)$ is not inductively free for $H_0 = \ker (x_1 - x_2)$ in $\CA$.
\end{example}

\subsection{Concentrated multiplicities}

In our final section 
we study 
multiplicities which are
concentrated at a single hyperplane. 
These were introduced by
Abe, Terao and Wakefield,
{\cite[\S 5]{abeteraowakefield:euler}.
	While in general, a multiarrangement  $(\CA, \mu)$ need not be free 
	for a free hyperplane arrangement $\CA$ and 
	an arbitrary multiplicity $\mu$, 
	e.g.~see \cite[Ex.~14]{ziegler:multiarrangements},
	for these concentrated multiplicities freeness is 
	also induced from the simple arrangement.
	It turns out that they are closely related to Ziegler's
	canonical multiplicity, 
	see Proposition \ref{prop:delta}. 
	
	\begin{defn}
		\label{def:delta}
		Let $\CA$ be a simple arrangement.
		Fix $H_0 \in \CA$ and $m_0 \in \BBZ_{\ge 1}$ and 
		define the 
		\emph{multiplicity $\delta$ concentrated at $H_0$}
		by
		\[
		\delta(H) := \delta_{H_0,m_0}(H) := 
		\begin{cases}
		m_0 & \text{ if } H = H_0,\\
		1   & \text{ else}.
		\end{cases}
		\]
	\end{defn}

	It turns out that both
	$\CA$ and $(\CA, \delta)$ 
	inherit freeness from one another:
	
	\begin{theorem}[{\cite[Thm.~1.7]{hogeroehrle:Ziegler}}]
		\label{thm:delta}
		Let $\CA$ be an arrangement.
		Fix $H_0 \in \CA$, $m_0 \in \BBZ_{\ge 1}$ and let 
		$\delta = \delta_{H_0,m_0}$ be 
		the multiplicity concentrated at $H_0$, 
		as in Definition \ref{def:delta}.
		Then 
		$\CA$ is free 
		with exponents
		$\exp \CA = \{1, e_2, \ldots, e_\ell\}$
		if and only if
		$(\CA, \delta)$ is free with exponents
		$\exp (\CA, \delta) = \{m_0, e_2, \ldots, e_\ell\}$. 
	\end{theorem}
		
	The following combines
	\cite[Prop.~5.2]{abeteraowakefield:euler}, parts of its proof
	and Theorem \ref{thm:zieglermulti}. 
	
	\begin{proposition}[{\cite[Prop.~2.14]{hogeroehrle:Ziegler}}]
		\label{prop:delta}
		Let $\CA$ be a free arrangement with exponents 
		$\exp \CA = \{1, e_2, \ldots, e_\ell\}$.
		Fix $H_0 \in \CA$, $m_0 \in \BBZ_{\ge 1}$ and let 
		$\delta = \delta_{H_0,m_0}$ be 
		as in Definition \ref{def:delta}.
		Let  $(\CA'', \delta^*)$ be the restriction with respect to $H_0$. 
		Then we have
		\begin{itemize}
			\item[(i)] 
			$(\CA, \delta)$ is free with exponents
			$\exp (\CA, \delta) = \{m_0, e_2, \ldots, e_\ell\}$;
			\item[(ii)] 
			$(\CA'', \delta^*) = (\CA'', \kappa)$ is free with exponents
			$\exp  (\CA'', \kappa) = \{e_2, \ldots, e_\ell\}$.
		\end{itemize}
	\end{proposition}
	
	Here is  
	the analogue of Theorem \ref{thm:main} in this setting which answers \cite[Ques.~1.8]{hogeroehrle:Ziegler}.
	
	\begin{corollary}
		\label{cor:main-delta}
		If $\CA$  is inductively free, then so is  
		$(\CA, \delta)$.
	\end{corollary}
	
	\begin{proof}
		Since $\CA$ is inductively free, so is $(\CA'', \kappa)$, by Theorem \ref{thm:main}. 
		The result now follows from induction on $m_0 \ge 1$,
		Proposition \ref{prop:delta}(ii) and a repeated 
		application of the addition part of Theorem \ref{thm:add-del}.
	\end{proof}

\bigskip {\bf Acknowledgments}:
This work was supported by DFG-Grant
RO 1072/21-1 (DFG Project number 494889912) to G.~R\"ohrle.


\bigskip

\bibliographystyle{amsalpha}

\newcommand{\etalchar}[1]{$^{#1}$}
\providecommand{\bysame}{\leavevmode\hbox to3em{\hrulefill}\thinspace}
\providecommand{\MR}{\relax\ifhmode\unskip\space\fi MR }
\providecommand{\MRhref}[2]{%
	\href{http://www.ams.org/mathscinet-getitem?mr=#1}{#2} }
\providecommand{\href}[2]{#2}


\end{document}